\pdfoutput=1
\documentclass{amsart}
\chardef\bslash=`\\

\usepackage{amssymb,amsmath,amsfonts,amsthm,epsfig,amscd,stmaryrd,enumitem}
\usepackage{stmaryrd}
\usepackage[all,cmtip,poly]{xy}
\usepackage{xcolor}
\usepackage{tikz}

\usepackage{silence}
\usepackage{microtype}

\newtheorem{theorem}[subsection]{Theorem}
\newtheorem{thm}[subsection]{Theorem}
\newtheorem{lemma}[subsection]{Lemma}

\newtheorem{prop}[subsection]{Proposition}

\newtheorem{defn}[subsection]{Definition}
\theoremstyle{remark}

\newtheorem{rem}[subsection]{Remark}

\numberwithin{equation}{subsection}

\newif\iffinalrun
\iffinalrun
\else
 \fi

\iffinalrun
  \newcommand{\need}[1]{}
  \newcommand{\mar}[1]{}
\else
  \newcommand{\need}[1]{{\tiny *** #1}}
  \newcommand{\mar}[1]{\marginpar{\raggedright\tiny  #1}}\fi

\renewcommand\mathbb{\mathbf}

\newcommand{\suffi}{{\operatorname{suffices}}}

\newcommand{\rec}{{\operatorname{rec}}}
\newcommand{\avoid}{\mathrm{avoid}} 
\newcommand{\barepsilon}{{\overline{\epsilon}}}

\newcommand{\rbar}{\overline{r}}

\renewcommand{\ell}{l}

\def\PSL{\mathrm{PSL}}
\def\PGL{\mathrm{PGL}}

\newcommand{\ad}{\operatorname{ad}}

\def\iso{\buildrel \sim \over \longrightarrow}

\newcommand{\G}{\cG}

\newcommand{\A}{\mathbf{A}}
\newcommand{\bA}{\ensuremath{\mathbf{A}}}
\newcommand{\CC}{{\mathbb C}}
\newcommand{\C}{\CC}
\newcommand{\bC}{\ensuremath{\mathbf{C}}}

\newcommand{\F}{\FF}
\newcommand{\FF}{{\mathbb F}}

\newcommand{\bQ}{\ensuremath{\mathbf{Q}}}
\newcommand{\Q}{\QQ}
\newcommand{\QQ}{{\mathbb Q}}

\newcommand{\Z}{\ZZ}
\newcommand{\ZZ}{{\mathbb Z}}

\newcommand{\cD}{{\mathcal D}}

\newcommand{\cG}{{\mathcal G}}

\newcommand{\cL}{{\mathcal L}}

\newcommand{\cM}{{\mathcal M}}
\newcommand{\cO}{{\mathcal O}}
\newcommand{\cR}{{\mathcal R}}

\newcommand{\CX}{{\mathcal{X}}}

\newcommand{\cY}{{\mathcal{Y}}}

\newcommand{\frakp}{\mathfrak{p}}

\newcommand{\ga}{\mathfrak{a}}
\newcommand{\gp}{\mathfrak{p}}
\newcommand{\frakl}{\mathfrak{l}}

\newcommand{\Fbar}{\overline{\F}}
\newcommand{\Qbar}{\overline{\Q}}

\newcommand{\Fp}{\F_p}

\newcommand{\Fl}{\F_{\ell}}

\newcommand{\Flbar}{\Fbar_{\ell}}

\newcommand{\Ql}{\Q_{\ell}}
\newcommand{\Qp}{\Q_p}

\newcommand{\Qlbar}{\Qbar_{\ell}}
\newcommand{\Qpbar}{\Qbar_p}

\DeclareMathOperator{\Aut}{Aut}

\DeclareMathOperator{\End}{End}

\DeclareMathOperator{\Gal}{Gal}
\newcommand{\GL}{\mathrm{GL}}
\newcommand{\GSp}{\mathrm{GSp}}
\DeclareMathOperator{\Hom}{Hom}

\DeclareMathOperator{\Ind}{Ind}

\DeclareMathOperator{\SL}{SL}

\DeclareMathOperator{\WD}{WD}

\newcommand{\Frob}{\mathrm{Frob}}
\newcommand{\HT}{\mathrm{HT}}
\newcommand{\nr}{\mathrm{nr}}

\newcommand{\rhobar}{\overline{\rho}}

\newcommand{\Res}{\operatorname{Res}}

\newcommand{\doubleslash}{/\kern-0.2em{/}}

\newcommand{\Fss}{\mathrm{F\mbox{-}ss}}
\usepackage{bookmark}
\hypersetup{hidelinks}

\title[Monodromy over CM fields]{Monodromy for some rank two Galois representations over CM fields}

\author[P.~Allen]{Patrick B. Allen} \email{patrick.allen@mcgill.ca}
 \address{Department of Mathematics and Statistics,
McGill University,
Montreal, Quebec H3A 0B9, Canada}

\author[J.~Newton]{James Newton} \email{j.newton@kcl.ac.uk}
\address{Department of Mathematics, King's College London, London WC2R 2LS, UK}

\thanks{P.A. \ was supported in part by Simons Foundation Collaboration Grant for Mathematicians 527275.}

\begin{document}

\begin{abstract}
We investigate local-global compatibility for cuspidal automorphic 
representations $\pi$ 
for $\operatorname{GL}_2$ over CM fields that are regular algebraic of weight $0$. 
We prove that 
for a Dirichlet density one set of primes $l$ and any $\iota : \Qbar_l \iso 
\C$, 
the $l$-adic Galois representation attached to $\pi$ and $\iota$ has nontrivial monodromy at any $v\nmid l$ in $F$ at which $\pi$ is special.
\end{abstract}

\maketitle
\setcounter{tocdepth}{1}
{\footnotesize
\tableofcontents
}

\section{Introduction}
Let $\pi$ be a regular algebraic cuspidal automorphic representation of $\GL_n$ over 
a CM field $F$. 
Choose a prime $l$ and an isomorphism $\iota : \Qbar_l \iso \C$. 
If $\pi$ is polarizable, then for any finite place $v$ of $F$, the Galois representation 
$r_\iota(\pi)$ attached to $\pi$ and $\iota$ 
satisfies local-global compatibility at $v$ \cite{BLGGT11, Caraianilnotp, Caraianilp, ht, shin, ty}.
The most subtle part is identifying the monodromy operator, 
the proofs of which rely on finding a base change of $r_\iota(\pi)$ or its 
tensor square
in the cohomology of a 
Shimura variety.

When $\pi$ is not polarizable, it should not be possible to find 
$r_\iota(\pi)$ itself in the 
cohomology of a Shimura variety (for precise statements, see 
\cite{Johansson_Thorne_cohom}). 
One can hope to access the direct sum of $r_\iota(\pi)$ and a twist 
of its conjugate dual via the cohomology of Shimura varieties, which is a basic 
starting point for the construction of $r_\iota(\pi)$ 
by Harris--Lan--Talor--Thorne \cite{hltt} as well as the alternate construction 
by 
Scholze \cite{Scholze-torsion}. 
These constructions use $l$-adic interpolation, so are well suited to keeping 
track 
of characteristic polynomials, and local-global compatibility was proved up to 
semisimplification by Varma \cite{ilavarma} 
for all $v \nmid l$.
But it doesn't seem possible to understand the monodromy operator in this way.
We overcome this problem in almost all cases of rank $2$ and weight $0$:

\begin{theorem}\label{theorem_main_from_intro}
Suppose that $F$ is a CM field and that $\pi$ is a regular
algebraic cuspidal automorphic representation of $\GL_2(\A_F)$ of
weight $0$.
There is a set of primes $l$ of Dirichlet density one
such that for any $\iota : \Qbar_l \iso \C$,
the $l$-adic Galois representation $r_\iota(\pi) : 
G_F \to \GL_2(\Qbar_l)$ attached to $\pi$ and $\iota$ satisfies 
\[ \iota \mathrm{WD}(r_\iota(\pi)|_{G_{F_v}})^{\Fss} \cong 
\rec_{F_v}(\pi_v \lvert\det\rvert^{-1/2}), \]
for all finite places $v\nmid l$ in $F$. 
\end{theorem}

We prove a more technical result, Theorem~\ref{thm:monodromynew} below, 
that applies to any prime $l$ and $\iota : \Qbar_l \iso \C$ to which 
we can apply an automorphy lifting theorem. 
The hypotheses necessary to apply the automorphy lifting theorem 
are known to hold for a density one set of primes (see Lemma~\ref{lemma:bad_set}), 
but should hold for all but finitely many (see Remark~\ref{rmk:good-primes}).

\subsection*{Method of proof} 
In light of Varma's results, we need to prove that 
if $v \nmid l$ is a finite place of $F$ at which $\pi$ is special, 
then $r_\iota(\pi)|_{G_{F_v}}$ has nontrivial monodromy. 
Results of a similar spirit were proved by one of us (J.N.) \cite{James-monodromy} 
in the context of Hilbert modular forms of partial weight one. 
In this situation, the Galois representations are also constructed by congruences, 
and one cannot realize these Hecke eigensystems in the Betti cohomology of a Shimura variety.
The proof relies on a $p$-adic version of Mazur's principle \cite{James-JL}. 

Another approach was developed by Luu \cite{Luu-lg}, 
relying on automorphy lifting theorems.
The basic idea in the context of $\GL_2$ is as follows. 
Assume that $\pi$ is a twist of the Steinberg representation at $v$. 
After a base change, we can assume that $\pi$ is an unramified twist 
of the Steinberg representation at $v$.
Now assume that the $l$-adic Galois representation 
$r_\iota(\pi)$ is unramified at $v$. 
Then so is its residual representation $\overline{r}_\iota(\pi)$, 
so we can hope to find a congruence to an automorphic representation $\pi_1$ 
such that $\pi_1$ is unramified at $v$. 
One can then apply automorphy lifting with the place $v$ 
left out of the ramification set to prove that $r_\iota(\pi) \cong 
r_\iota(\pi_2)$ 	
for some automorphic representation $\pi_2$ that is unramified at $v$. 
This contradicts strong multiplicity one. 

The main ingredient needed to execute this strategy in the present context is an 
automorphy lifting theorem of \cite{10author}, recalled in 
Theorem~\ref{thm:main_automorphy_lifting_theorem} below. 
However, there is still a subtlety that needs to be overcome: 
we need to produce a congruence to the automorphic representation 
$\pi_1$ that is unramified at $v$. 
In the situations where the Galois representation in question 
does not appear in the Betti cohomology of a Shimura variety, 
these congruences don't always exist,
see~\cite[\S7.4.1, \S7.4.2]{Calegari-Venkatesh} for examples of level lowering 
congruences to torsion classes which do not have a characteristic 0 lift at the 
lower level. To get around this problem here, 
we use Taylor's potential automorphy method to first prove
(see Theorem~\ref{potaut} for a more precise statement):

\begin{thm}
Suppose that $F$ is a CM field and $l$ is an odd prime unramified in $F$. 
Let $\rhobar : G_F \to \GL_2(k)$ be a continuous representation with $k/\F_l$ 
finite such that $\det(\rhobar) = \overline{\epsilon}_l^{-1}$ 
and such that for each $w|l$,  
$\rhobar|_{G_{F_w}}$ admits a crystalline lift with all 
labelled Hodge--Tate weights equal to $\{0,1\}$

Then there is a CM extension $F'/F$ such that $\rhobar|_{G_{F'}}$ 
arises from a regular algebraic weight $0$ cuspidal automorphic 
representation $\pi_1$ of $\GL_2(\A_{F'})$ that we can assume is 
unramified above our fixed $v \nmid l$ in $F$.
\end{thm}

Applying the automorphy lifting theorem, we deduce that 
$r_\iota(\pi)|_{G_{F'}}$ arises from a cuspidal automorphic repesentation 
$\pi_2$ of $\GL_2(\A_{F'})$ that is unramified at all places above $v$.
We can no longer use multiplicity one, as this would require knowing the base 
change of $\pi$ to $F'$ exists, and $F'/F$ may not be solvable. 
However, by Varma's results, we know the Frobenius eigenvalues of 
$r_\iota(\pi)$ at $v$ and thus at any place above $v$ in $F'$.
This together with the fact that $\pi_2$ is unramified above $v$ 
contradicts the genericity of $\pi_2$.

Naturally, the automorphy lifting theorem we use contains several technical assumptions, reflected in the statement of Theorem~\ref{thm:monodromynew}. 
These technical assumptions should hold for all but finitely many primes $l$, but we do not know how to prove this. 
Using results of Larsen \cite{Larsen}, one can show that they hold on a density 
one set (see Lemma~\ref{lemma:bad_set}), resulting in 
Theorem~\ref{theorem_main_from_intro}. 

Let us finally remark on the restriction to rank $2$ and weight $0$ in Theorems~\ref{theorem_main_from_intro} and \ref{thm:monodromynew}. 
These restrictions come from the combination of automorphy lifting theorems and 
potential automorphy theorems available to us in the non-polarizable case over CM fields. 
(There is also a simplification in rank $2$ that there are only two possible conjugacy classes for the monodromy operator, but we do not believe that this is a serious issue.)
Common to any potential automorphy theorem is some moduli space of motives in which one realizes a fixed mod $l$ Galois representation. 
In \cite[\S7.2.5]{10author}, a moduli space of elliptic curves is used, and this could be used to prove a version Theorem~\ref{thm:monodromynew} with the additional assumption that the prime $l$ splits completely in the coefficient field $M_\pi$ of $\pi$. 
This in turn would yield a version of Theorem~\ref{theorem_main_from_intro} valid for a positive density set of primes $l$. 
In order to allow more general primes $l$ in Theorem~\ref{thm:monodromynew} and obtain the density one set in Theorem~\ref{theorem_main_from_intro}, we use a moduli space of Hilbert--Blumenthal abelian varieties, as in Taylor's original work on the subject \cite{tay-fm}. 
In either case, we are confined to rank $2$, and since the Fontaine--Laffaille automorphy lifting theorem that we use does not allow a change of weight, we are also confined to weight $0$.

Still working in rank $2$, it seems possible to remove the weight $0$ assumption in Theorem~\ref{thm:monodromynew} at the cost of imposing an ordinarity assumption on $\pi$ at the prime $l$ by using the ordinary automorphy lifting theorem of \cite{10author}, which does allow change of weight. 
This is currently being investigated by Yuji Yang.

In higher rank, there are robust potential automorphy theorems in the polarizable\ case \cite{HSBT,blght,BLGGT} using the so-called Dwork family of motives. 
If theorems similar to these were proved in the non-polarizable case, then we 
believe that the methods of this paper could be used to prove versions of 
Theorem~\ref{theorem_main_from_intro} and \ref{thm:monodromynew} in higher rank.

\subsection*{Notation}
For a field $F$, we let $\overline{F}$ denote a separable closure and 
$G_F = \Gal(\overline{F}/F)$ the absolute Galois group. If $F$ is a CM number 
field, then we write $F^+$ for its maximal totally real subfield. 
CM number fields are always assumed to be imaginary.

Let $F$ be a number field. 
If $v$ a finite place of $F$, $l$ is a prime, and 
$r : G_{F_v} \to \GL_2(\Qbar_l)$ is a continuous representation, 
we let $\WD(r)^{\Fss}$ be the associated 
Frobenius semisimple Weil--Deligne representation. 
If $\iota : \Qbar_l \iso \C$ is an isomorphism of fields, 
we let $\iota\WD(r)^{\Fss}$ denote its extension of scalars 
to $\C$ via $\iota$. 
We write $\rec_{F_v}$ for the local Langlands correspondence of \cite{ht}.

Let $\pi$ be a regular algebraic cuspidal automorphic representation 
of $\GL_2(\A_F)$. 
We say that $\pi$ has \emph{weight} $0$ if it has the same 
infinitesimal character as the trivial (algebraic) representation 
of $\Res_{F/\Q}\GL_2$. 
We let $M_\pi \subset \C$ denote the coefficient field of $\pi$; 
it is the fixed field of 
$\{\sigma \in \Aut(\C) : {}^\sigma \pi^\infty \cong \pi^\infty\}$. 
If $l$ is a prime and $\iota : \overline{\Q}_l \iso \C$ is an 
isomorphism of fields, we let $r_\iota(\pi) : G_F \to \GL_2(\Qbar_l)$ 
be the $l$-adic Galois representation attached to $\pi$ and $\iota$ 
by Harris--Lan--Taylor--Thorne \cite{hltt}. 
It is characterized by the property that if $p \ne l$ is a prime 
above which $\pi$ and $F$ are unramified and $v|p$ in $F$, then
\[ \iota\WD(r_\iota(\pi)|_{G_{F_v}})^{\Fss} \cong 
\rec_{F_v}(\pi_v\lvert\det\rvert^{-1/2}).\]
The isomorphism $\iota$ induces a prime $\lambda|l$ in $M_\pi$ 
and an algebraic closure $\overline{M}_{\pi,\lambda} = \Qbar_l$ of 
the completion $M_{\pi,\lambda}$, and we also write 
$r_{\pi,\lambda} : G_F \to \GL_2(\overline{M}_\lambda)$ for $r_\iota(\pi)$ 
in this case. 
Conversely, given $\lambda|l$ in $M_\pi$, an algebraic closure 
$\overline{M}_{\pi,\lambda}$ of $M_{\pi,\lambda}$, and an isomorphism 
$\iota : \overline{M}_{\pi,\lambda} \iso \C$, 
we obtain $r_{\pi,\lambda} = r_\iota(\pi)$ 
by identifying $\overline{M}_{\pi,\lambda}$ with $\Qbar_l$.

We let $\epsilon_l$ denote the $l$-adic cyclotomic character. 
We normalize our Hodge--Tate weights so that $\epsilon_l$ 
has all labelled Hodge--Tate weights equal to $-1$. 
We let $\zeta_l$ denote a primitive $l$th root of unity. 

Let $F$ and $M$ be number fields. 
If $A$ is an abelian variety over $F$ equipped with an 
embedding of rings $\cO_M \hookrightarrow \End(A)$ and 
$\frakl$ is a prime of $M$, then we let $r_{A,\frakl}$ 
denote the representation of $G_F$ on 
$T_\frakl(A)\otimes_{\cO_{M,\frakl}} \overline{M}_\frakl$.

\subsection*{Acknowledgements}
We would like to thank Ariel Weiss, Yuji Yang, and the anonymous referees for comments on 
earlier versions of this paper. 
We would like to thank Frank Calegari, Ana Caraiani, David Helm, 
Toby Gee, Bao Le Hung, Peter Scholze, Richard Taylor, and Jack Thorne 
for many helpful conversations. 
This project grew out of our joint work \cite{10author}, 
and the intellectual debt this work owes to \cite{10author} 
will be clear to the reader.

\section{Automorphy of compatible systems}\label{sec:aut_lifting}

The crucial ingredient we need for the results of this paper is 
the following automorphy lifting theorem over CM fields,
which is a special case of \cite[Theorem 6.1.1]{10author} (the notions 
of enormous and decomposed generic will be recalled after the statement of the 
theorem):

\begin{theorem}\label{thm:main_automorphy_lifting_theorem}
	Let $F$ be a CM field and let $\rho : G_F \to \GL_2(\overline{\bQ}_l)$ 
    be a continuous representation satisfying the following conditions:
	\begin{enumerate}
		\item $\rho$ is unramified almost everywhere.
		\item For each place $v | l$ of $F$, the representation 
		$\rho|_{G_{F_v}}$ is crystalline with labelled Hodge--Tate weights all 
		equal to $\{0,1\}$. The prime $l$ is unramified in $F$.
		\item\label{hyp:residual_image}
		$\overline{\rho}$ is decomposed generic and 
		$\overline{\rho}|_{G_{F(\zeta_l)}}$
is absolutely irreducible with enormous image. There exists 
		$\sigma \in G_F - G_{F(\zeta_l)}$ such that $\overline{\rho}(\sigma)$ 
		is a scalar. We have $l \ge 5$.
		\item There exists a cuspidal automorphic representation $\pi$ of 
		$\GL_2(\bA_F)$ satisfying the following conditions:
		\begin{enumerate}
			\item $\pi$ is regular algebraic of weight $0$.
			\item There exists an 
			isomorphism $\iota : \overline{\bQ}_l \iso 
			\bC$ such that $\rhobar \cong \overline{r}_\iota(\pi)$.
\item If $v | l$ is a place of $F$, then $\pi_v$ is unramified. 
		\end{enumerate}
	\end{enumerate}
	Then $\rho$ is automorphic: there exists a cuspidal automorphic 
	representation $\Pi$ of $\GL_2(\bA_F)$ of weight $0$ such that $\rho 
	\cong r_\iota(\Pi)$. Moreover, if $v$ is a finite place of $F$ and 
	either 
	$v | l$ or both $\rho$ and $\pi$ are unramified at $v$, then $\Pi_v$ is 
	unramified.
\end{theorem}

Let $\rhobar : G_F \to \GL_2(\Fbar_l)$ be a continuous representation and 
let $\ad^0$ denote the set of trace zero matrices in $\operatorname{M}_{2\times 2}(\Fbar_l)$. 
The image $H = \rhobar(G_F)$ of $\rhobar$ is called 
\emph{enormous} if it satisfies the following 
(c.f \cite[Definition~6.2.28 and Lemma~6.2.29]{10author}):
\begin{enumerate}
	\item $H$ has no nontrivial $l$-power order quotient.
	\item $H^0(H, \ad^0) = H^1(H, \ad^0) = 0$.
	\item For any simple $\Fbar_l[H]$-submodule $W \subseteq \ad^0$, 
	there is a regular semisimple $h \in H$ such that $W^h \ne 0$.
\end{enumerate}
We say a prime $p \ne l$ is \emph{decomposed generic for} $\rhobar$ if it
splits completely in $F$ and for any $v|p$ in $F$, $\rhobar$ is unramified at $v$ and the eigenvalues $\alpha_v$, $\beta_v$ of $\rhobar(\Frob_v)$ satisfy 
$\alpha_v\beta_v^{-1} \notin \{1, p, p^{-1}\}$ 
(c.f. \cite[Definition~2.2.4]{10author}).
We say that $\rhobar$ is \emph{decomposed generic} if there is a prime $p \ne 
l$ that is decomposed generic for $\rhobar$.

\begin{lemma}\label{lemma:enormours}
    Let $F$ be a number field and let $\rhobar : G_F \to 
    \GL_2(\Fbar_l)$ be a continuous representation with $l > 5$. 
    If $\rhobar(G_F) \supseteq \SL_2(\F_l)$, then $\rhobar(G_{F(\zeta_l)})$ 
    is enormous.
\end{lemma}

\begin{proof}
By \cite[Lemma~3.2.3]{geenew}, and since we have assumed that $l > 5$, 
it suffices to show that $\rhobar|_{G_{F(\zeta_l)}}$ is absolutely irreducible. 
This follows from the assumption that $\rhobar(G_F) \supseteq \SL_2(\F_l)$ and 
the fact that $\SL_2(\F_l)$ is perfect when $l > 3$.
\end{proof}

\begin{lemma}\label{lemma:generic}
    Let $F/\Q$ be a finite Galois extension and let $\rhobar : G_F \to 
    \GL_2(\Fbar_l)$ be a continuous representation with $l > 3$. 
    If $\rhobar(G_F) \supseteq \SL_2(\F_l)$, then $\rhobar$ 
    is decomposed generic.
\end{lemma}

\begin{proof}
This is contained in the proof of \cite[Lemma~7.1.5]{10author}. 
For the convenience of the reader, we give the details. 
It suffices to prove $\rhobar$ is decomposed 
generic after replacing $F$ with some finite extension. 
Conjugating $\rhobar$ if necessary, \cite[Theorem~2.47(b)]{DDT} 
implies that there is an extension $F'/F$ of degree at most $2$ such that $\rhobar|_{G_{F'}}$ has projective image $\PSL_2(k)$, for some finite $k/\Fl$, 
which is simple since $l > 3$. 
If $\widetilde{F'}$ is the Galois closure of $F'/\Q$, 
then $\widetilde{F'}/F$ is abelian since $F/\Q$ is Galois and $F'/F$ is abelian. 
We can thus replace $F$ with $\widetilde{F'}$. 

Let $H/F$ be the extension cut out by the projective 
image of $\rhobar$ and let $\widetilde{H}/H$ be 
the Galois closure of $H/\Q$. 
Since $\PSL_2(k)$ is simple, Goursat's lemma implies that 
$\Gal(\widetilde{H}/F) \cong \PSL_2(k)^n$ for some $n\ge 1$. 
Fix a non-identity semisimple $g \in \PSL_2(k)$. 
By Chebotarev density, we can find a prime $w$ of $F$ such that $\Frob_w$ in 
$\Gal(\widetilde{H}/F)$ is $(g,g,\ldots,g)$.  
Moreover, we can assume that the residue field at $w$ is $\F_p$ with $p \ne l$ unramified in $F$, since such primes have Dirichlet density one in $F$. 
Since $F/\Q$ is Galois, $p$ is totally split in $F$. 
Take $v|p$ in $F$.
Since $\PSL_2(k)$ is simple, the only normal subgroups of $\PSL_2(k)^n$ 
are of the form $\PSL_2(k)^I$ for $I \subset\{1,\ldots,n\}$, 
so $\Aut(\PSL_2(k)^n) = \Aut(\PSL_2(k))^n \rtimes S_n$, and 
$\Aut(\PSL_2(k)) = \PGL_2(k) \rtimes \Gal(k/\F_l)$ by \cite{dieudonne_aut_classical}. 
So the image of $\rhobar(\Frob_v)$ in $\PSL_2(k)$ is $\tau(g)$ for some $\tau \in \PGL_2(k) \rtimes \Gal(k/\F_l)$. 
In particular, it is semisimple and $\ne 1$, so $\rhobar(\Frob_v)$ has distinct eigenvalues. 
As $\zeta_l \in F$, $\rhobar$ is decomposed generic.    
\end{proof}

\begin{lemma}\label{lemma:scalar}	
	Suppose $l > 3$, let $F$ be a number field in which $l$ is unramified and 
	let $\rhobar : G_F \to \GL_2(\Fbar_l)$ be a continuous representation such 
	that $\rhobar(G_F) \supseteq \SL_2(\F_l)$. 
	Then there exists $\sigma \in G_F - G_{F(\zeta_l)}$ such that 
	$\rhobar(\sigma)$ is a scalar.
\end{lemma}
\begin{proof}
	This is again contained in the proof of \cite[Lemma~7.1.5]{10author}, but 
	we give the details. By \cite[Theorem~2.47(b)]{DDT}, the image of 
	$\rhobar(G_F)$ in $\PGL_2(\Fbar_l)$ is conjugate to $\PSL_2(k)$ or 
	$\PGL_2(k)$ for some finite subfield $k \subset \Fbar_l$. This projective 
	image is isomorphic to the image of $\ad\rhobar$. Since $l$ is unramified 
	in $F$, we have $\Gal(F(\zeta_l)/F) \cong (\ZZ/l\ZZ)^\times$. If $l > 3$, 
	neither $\PSL_2(k)$ nor $\PGL_2(k)$ admits $(\ZZ/l\ZZ)^\times$ as a 
	quotient, so we deduce that $F(\zeta_l)$ is not contained in 
	$\overline{F}^{\ker\ad\rhobar}$. This amounts to the existence of the 
	desired element $\sigma$.
\end{proof}

\subsection{Compatible systems}
We follow the terminology of \cite[\S7.1]{10author}.
For $F$ a number field, a \emph{rank} $n$ 
\emph{extremely weakly compatible system of Galois representations} is a tuple
\[ \cR = (M, S, \{Q_v(X)\}, \{r_\lambda\}) \]
where
\begin{itemize}
    \item $M$ is a number field;
    \item $S$ is a finite set of primes of $F$;
    \item for each prime $v \notin S$ of $F$, $Q_v(X) \in M[X]$ is a monic polynomial of degree $n$;
    \item for each prime $\lambda$ of $M$, $r_\lambda : G_F \to 
    \GL_n(\overline{M}_\lambda)$ is a continuous semisimple representation such 
    that for every prime $v$ of $F$ with $v \notin S$ and not dividing the 
    residual characteristic of $\lambda$, $r_\lambda$ is unramified at $v$ and 
    $r_\lambda(\Frob_v)$ has characteristic polynomial $Q_v(X)$.
\end{itemize} 
There are obvious notions of direct sums, duals, tensor products, 
inductions, etc.~for extremely weakly compatible systems. 
In particular, we have a rank one extremely weakly compatible system 
$\det\cR$ obtained by taking determinants of the $r_\lambda$. 
By \cite{henniart}, $\det(r_\lambda)$ is de~Rham for each $\lambda$, 
and for any embedding $\tau : F \hookrightarrow \overline{M}_\lambda$, 
$\HT_\tau(\det(r_\lambda))$ is independent of $\lambda$. 

We say $\cR$ is \emph{irreducible} if every $r_\lambda$ is irreducible (in the 
case of rank two, this is equivalent to any one $r_\lambda$ being irreducible, 
see \cite[Lemma~7.1.1]{10author}). 
We say that $\cR$ is strongly irreducible if 
$\cR|_{G_{F'}}$ is irreducible for any finite extension $F'/F$. 
The following lemma is contained in \cite[Lemma 7.1.3]{10author} (which relies 
on a result of Larsen \cite{Larsen}).

\begin{lemma}\label{lemma:open-im}
Let $F$ be a number field and let
\[ \cR = (M, S, \{Q_v(X)\}, \{r_\lambda\}) \]
be a strongly irreducible rank two extremely weakly compatible system.
Then there is a set $\cL$ of rational primes with Dirichlet density one 
such that for all $l \in \cL$ and $\lambda | l$ in $M$, 
the image of $\rbar_\lambda$ contains a conjugate of $\SL_2(\F_l)$.
\end{lemma}

By the main theorem of \cite{hltt} (together with \cite{ilavarma}, to get 
local--global compatibility at all places where $\pi$ is unramified), if $\pi$ 
is a regular algebraic cuspidal automorphic representation of $\GL_2(\A_F)$ 
with $F$ a CM field, then we have a rank two extremely weakly compatible system 
\[ \cR_\pi = (M_\pi, S_\pi, \{Q_{\pi,v}(X)\}, \{r_{\pi,\lambda}\}) \]
with
\begin{itemize}
    \item $M_\pi \subset \C$ the coefficient field of $\pi$;
    \item $S_\pi$ the set of primes of $F$ at which $\pi$ is ramified;
    \item $Q_{\pi,v}(X)$ is the characteristic polynomial of $\rec_{F_v}(\pi_v\lvert\det\rvert^{-1/2})(\Frob_v)$.
\end{itemize}

\begin{lemma}\label{lemma:steinberg}
Let $F$ be a CM field and let $\pi$ be a regular algebraic cuspidal 
automorphic representation of $\GL_2(\A_F)$ of weight $0$.
If $\pi$ is a twist of Steinberg at some finite place of $F$, 
then $\cR_\pi$ is strongly irreducible.
\end{lemma}

\begin{proof}
It is well known that $\cR_\pi$ is irreducible. 
Since $\pi$ has weight $0$, there is a finite order character 
$\chi : G_F \to M_\pi^\times$ such that $\det(r_{\pi,\lambda}) = 
\chi\epsilon_l^{-1}$ 
for any $l$ and $\lambda|l$. 
It then follows from \cite[Lemma~7.1.2]{10author} that either $\cR_\pi$ 
is strongly irreducible or there is a quadratic extension $K/F$ 
and an extremely weakly compatible system $\CX$ of characters of $G_K$ 
such that $\cR = \Ind_{G_K}^{G_F} \CX$. 
In the latter case, the system $\CX$ is the extremely weakly compatible system associated 
to a Hecke character $\psi : K^\times \backslash \bA_K^\times \to \C^\times$, 
and we deduce that $\pi$ is the automorphic induction of $\psi$. 
Such a $\pi$ cannot be a twist of the Steinberg representation at any finite 
place.
\end{proof}

\begin{theorem}\label{thm:lg-at-p}
    Let $F$ be a CM field and let $\pi$ be a regular algebraic cuspidal 
    automorphic representation of $\GL_2(\A_F)$ of weight $0$.
    Let $\lambda|l$ be a prime of the coefficient field $M_\pi \subset \C$ of 
    $\pi$, and let $r_{\pi,\lambda} : G_F \to 
    \GL_2(\overline{M}_{\pi,\lambda})$ 
    be the $\lambda$-adic Galois representation attached to $\pi$. 
    Assume that the residual representation $\overline{r}_{\pi,\lambda}$ is 
    absolutely irreducible and decomposed generic. 
    Assume also that $l$ is unramified in $F$ and lies under no prime
    at which $\pi$ is ramified. 
    Then for any $v|l$ in $F$, $r_{\pi,\lambda}|_{G_{F_v}}$ is crystalline 
    with all labelled Hodge--Tate weights equal to $\{0,1\}$.
\end{theorem}

\begin{proof}
    The deduction of the theorem from \cite[Theorem~4.5.1]{10author}, 
    is contained in \cite[Lemma~7.1.8]{10author}. 
    We give a sketch.
Fix $v|l$ in $F$. We can replace $F$ with a finite solvable 
    extension in which $v$ splits completely. 
    Doing so, we may assume the following:
    \begin{itemize}
        \item $F = F^+ F_0$ with $F^+$ totally real and $F_0$ an imaginary 
        quadratic field in which $l$ splits.
        \item There are at least three places above $l$ in $F^+$, 
        and letting $\overline{v}$ be the place of $F^+$ below $v$, 
        there is $\overline{v}' \ne \overline{v}$ dividing $l$ in $F^+$ 
        such that
        \[ \sum_{\overline{v}'' \ne \overline{v},\overline{v}'} 
        [F_{\overline{v}''}^+:\Q_p] > \frac{1}{2}[F^+:\Q_p].\]
    \end{itemize}
    Then $r_{\pi,\lambda}$ is a $\overline{M}_\lambda$-point of the Hecke 
    algebra $\mathbf{T}^S(K,0)$ of \cite[Theorem~4.5.1]{10author}
    for appropriate choices of a finite set of primes $S$ of $F$ 
    and a level subgroup $K \subset \GL_2(\A_F^\infty)$, 
    from which the theorem follows.
\end{proof}

\begin{lemma}\label{lemma:bad_set}
    Let $F$ be a CM field, let $\pi$ be a regular algebraic cuspidal 
    automorphic representation of $\GL_2(\A_F)$ of weight $0$, and let $M_\pi \subset \C$ 
    be its coefficient field. 
    Assume that $\pi_v$ is a twist of the Steinberg representation at some finite 
    place $v$ of $F$.
    Then there is a set $\cL$ of rational primes with Dirichlet density one 
    such that for all $l \in \cL$ and all $\lambda |l$ in $M_\pi$, 
    the Galois representation     $r_{\pi,\lambda} : G_F \to 
    \GL_2(\overline{M}_\lambda)$ 
    satisfies the following:
    \begin{enumerate}
        \item For each place $v | l$ of $F$, the representation $r_{\pi,\lambda}|_{G_{F_v}}$ is 
        crystalline with labelled Hodge--Tate weights all equal to $\{0,1\}$. 
        The prime $l$ is unramified in $F$.
        \item $\rbar_{\pi,\lambda}$ is absolutely irreducible and
        decomposed generic. The image of 
        $\rbar_{\pi,\lambda}|_{G_{F(\zeta_l)}}$ 
        is enormous. 
        There exists $\sigma \in G_F - G_{F(\zeta_l)}$ such that 
        $\rbar_{\pi,\lambda}(\sigma)$ is a scalar. We have $l \ge 5$.
    \end{enumerate}
\end{lemma}

\begin{proof}
Let $\cR = (M_\pi, S_\pi, \{Q_{\pi,v}\}, \{r_{\pi,\lambda}\})$ 
be the rank $2$ extremely weakly compatible system attached to $\pi$. 
By Lemma~\ref{lemma:steinberg}, $\cR$ is strongly irreducible. 
Let $\widetilde{F}$ be the Galois closure of $F/\Q$. 
By Lemmas~\ref{lemma:enormours}, \ref{lemma:generic}, 
\ref{lemma:scalar}, and Theorem~\ref{thm:lg-at-p}, 
it suffices to show that there is a Dirichlet density one 
set $\cL$ of primes $l$, unramified in $F$, such that for all $l \in \cL$ 
and $\lambda|l$ in $M_\pi$, the following hold:
\begin{enumerate}[label=(\alph*)]
    \item\label{sl2} $\rbar_{\pi,\lambda}(G_{\widetilde{F}})$ contains 
    a conjugate of $\SL_2(\F_l)$. 
\item\label{local} $l > 5$ and lies under no prime at which $\pi$ is 
    ramified. 
\end{enumerate}

The restriction $\cR|_{G_{\widetilde{F}}}$ is again strongly irreducible, 
so Lemma~\ref{lemma:open-im} implies that there is a 
Dirichlet density one set $\cL'$ of primes $l$ such that 
$\rbar_{\pi,\lambda}(G_{\widetilde{F}})$ contains a conjugate of 
$\SL_2(\F_l)$ for any $l \in \cL'$ and $\lambda|l$ in $M_\pi$.

We obtain $\cL$ by removing from $\cL'$ the finite set of primes $l$ satisfying 
either $l \le 5$, 
$l$ ramifies in $F$, or $l$ lies under a place at which $\pi$ is ramified. 
\end{proof}

\begin{rem}\label{rmk:good-primes}
It should be apparent from the proof of Lemma~\ref{lemma:bad_set} that 
the statement could be improved from ``Dirichlet density one" to 
``all but finitely many," provided one could prove that 
the image of $\rbar_{\pi,\lambda}$ contains a 
conjugate of $\SL_2(\Fl)$ for all but finitely many places $\lambda$. 
This would imply a similar strengthening of 
Theorem~\ref{theorem_main_from_intro}. 
In fact, it would suffice to know that $\rbar_{\pi,\lambda}|_{G_{F(\zeta_l)}}$ 
is absolutely irreducible and decomposed generic for all but finitely many $\lambda$. 
Using the purity of $r_{\pi,\lambda}$ (see \cite[Corollary~7.1.12]{10author}), 
it is not hard to see that the decomposed generic condition holds for all but 
finitely many $\lambda$, so the main obstruction is showing residual irreducibility.
This can be shown 
(see, for example, the main theorem of \cite{Hui-Larsen}), provided we know 
that the representations $r_{\pi,\lambda}$ are crystalline 
with the correct Hodge--Tate weights (without assuming 
residual irreducibility as in Theorem \ref{thm:lg-at-p}). Such a 
crystallinity result has been proven by Mok \cite{mok-galois}, under some 
technical hypotheses and using Arthur's classification for 
automorphic 
representations of $\GSp_4$ (see \cite{MR2058604,2018arXiv180703988G}). 
\end{rem}

\section{Potential automorphy}
We begin by recalling a theorem of Moret--Bailly \cite{mb}:

\begin{prop}\label{moretbailly}
Let $L$ be a totally real number field and let $S_1 \coprod S_2$ be a finite 
set of finite places of $L$. 
Suppose that $X/L$ is a smooth, geometrically connected variety. 
Suppose also that $X(L_v)\ne \emptyset$ for all real places $v$ of $L$, that 
$\Omega_v \subset X(L_v^{\nr})$ is a non-empty open (for the $v$-topology) 
$\Gal(L_v^{\nr}/L_v)$-invariant subset for the places $v \in S_1$ and that 
$\Omega_v\subset X(\overline{L}_v)$ is a non-empty open 
$\Gal(\overline{L}_v/L_v)$-invariant subset for the places $v \in S_2$. Suppose 
finally that $L^{\avoid}/L$ is 
a 
finite extension.

Then there is a finite Galois totally real extension $L_1/L$ and a point $P \in 
X(L_1)$ such that \begin{itemize}
\item $L_1/L$ is linearly disjoint from $L^{\avoid}/L$
\item every place $v \in S_1$ is unramified in $L_1$  and if $w$ is a prime of 
$L_1$ above $v$ then $P \in \Omega_v\cap X(L_{1,w})$
\item if $w$ is a prime of $L_1$ above $v \in S_2$ then $P \in \Omega_v\cap 
X(L_{1,w})$.
\end{itemize}
\end{prop}
\begin{proof}
Our precise statement is a special case of \cite[Prop.~2.1]{HSBT}.
\end{proof}

We also recall a result on potential modularity of elliptic curves which is 
essentially contained in \cite{tay-fm2}:
\begin{prop}\label{ecpotaut}
Suppose that $E/\Q$ is a non-CM elliptic curve, and that $\cL$ is a finite set 
of rational primes at which $E$ has good reduction.
Suppose also that 
$L_1^{\avoid}/\Q$ is a finite extension.

Then we can find \begin{itemize}
	\item a finite Galois extension $L_2^{\avoid}/\Q$ linearly disjoint from 
	$L_1^{\avoid}$ over $\Q$ and
	\item a finite totally real Galois extension $L^{\suffi}/\Q$, unramified 
	above $\cL$ such that $L^{\suffi}$ is linearly disjoint from 
	$L_1^{\avoid}L_2^{\avoid}$ over $\Q$
\end{itemize}

such that for any finite totally real extension $L_2/L^{\suffi}$ which is 
linearly disjoint from $L_2^{\avoid}$ over $\Q$, there is a regular algebraic  
cuspidal automorphic representation $\pi$ of $\GL_2(\A_{L_2})$ of weight $0$ 
such that for every rational prime $l$ and any $\iota: \overline{\Q}_l \cong 
\C$ we have \[r_{\iota}(\pi)\cong r_{E,l}^\vee|_{G_{L_2}}.\] 
Moreover, $\pi$ is 
unramified above any prime where $E$ has good reduction.
\end{prop}
\begin{proof}
Our precise statement is a special case of \cite[Corollary 7.2.4]{10author}.
\end{proof}

\subsection{HBAVs}
Let $M$ be a totally real number field and let $S$ be a scheme. For an Abelian 
scheme $A/S$ equipped with a ring embedding $\iota: \cO_M\hookrightarrow 
\End(A/S)$ we denote by $(\cM_A,\cM_A^+)$ the module $\cM_A$ of $\cO_M$-linear, 
symmetric homomorphisms from $A$ to $A^\vee$, with its positive cone $\cM_A^+$ 
of polarizations.  

\begin{defn}
	An $M$-HBAV over $S$ is a pair $(A,\iota)$ as above, such that the natural 
	map $A\otimes_{\cO_M}\cM_A\to A^\vee$ is an isomorphism. 
	
	For a non-zero fractional ideal $\mathfrak{c} \subset M$, a 
	$\mathfrak{c}$-polarization of an $M$-HBAV $A$ is an isomorphism $j: 
	\mathfrak{c} \iso \cM_A$ of $\cO_M$-modules with $j(\mathfrak{c}^+) = 
	\cM_A^+$, where $\mathfrak{c}^+$ denotes the totally positive 
	elements of $\mathfrak{c}$.
\end{defn}
\begin{rem}\label{polnremark}
	When the fractional ideal $\mathfrak{c}$ contains $\cO_M$ the map $x 
	\mapsto x\otimes 1$ induces an isomorphism $A/A[\mathfrak{c}^{-1}] \cong 
	A\otimes_{\cO_M}\mathfrak{c}$. It follows that specifying a 
	$\mathfrak{c}$-polarization $j$ of an $M$-HBAV $A$ is equivalent to 
	specifying 
	a single $\cO_M$-linear polarization $\lambda: A \to A^{\vee}$ with kernel 
	$A[\mathfrak{c}^{-1}]$. The polarization $\lambda$ corresponds to $j(1)$. 
	See \cite[\S2.6]{DP}.
\end{rem}

If $(A,\iota,j)$ is an $M$-HBAV equipped with a $\mathfrak{c}$-polarization and 
$\ga$ is an ideal in $\cO_M$ we obtain perfect $\cO_M/\ga$-bilinear alternating 
pairings (see \cite[\S2.12]{DP}) \begin{equation}\label{pairing}A[\ga] \times 
A[\ga] \to 
(\mathfrak{c}^{-1}\cD^{-1}\otimes \mu_{N\ga})[\ga],\end{equation} 
with $\cD \subseteq \cO_M$ the different of $M/\Q$ and $N\ga = \lvert 
\cO_M/\ga\rvert$ the norm of $\ga$, and isomorphisms  
\[A[\ga]\otimes_{\cO_M}\mathfrak{c} \to A[\ga]^\vee.\] We refer to these 
alternating pairings as the Weil pairing for an HBAV with a fixed 
$\mathfrak{c}$-polarization.

We will also use the following related construction. 
Let $K$ be either $F$ or its completion at some prime. 
Let $\frakl$ be a prime of $M$ of residual characteristic $l$, 
and let $\rbar : G_K \to \GL_2(k_\frakl)$ 
be a continuous representation with $\det\rbar = \overline{\epsilon}_l$. 
Letting $V_{\rbar}$ be the \'etale $k_\frakl$-vector space scheme over $K$ defined by $\rbar$, 
the standard symplectic pairing on $V_{\rbar}$ is an $\cO_M$-bilinear perfect pairing
\begin{equation}\label{symplectic_pairing}
  V_{\rbar} \times V_{\rbar} \to (\cO_M \otimes\mu_l)[\frakl]
\end{equation}
that induces an isomorphism $V_{\rbar} \otimes \cD^{-1} \cong V_{\rbar}^\vee$, 
with $V_{\rbar}^\vee$ the Cartier dual $\Hom(V_{\rbar},\G_m)$ of the group 
scheme $V_{\rbar}$ (which naturally inherits the structure of $k_\frakl$-vector 
space scheme).
If further $\frakl$ is unramified in $M$, then $x \mapsto x \otimes 1$ defines an 
isomorphism $V_{\rbar} \cong V_{\rbar}\otimes \cD^{-1}$ and the symplectic 
pairing defines an canonical isomorphism $V_{\rbar} \cong V_{\rbar}^\vee$. 

If $\frakl$ is a prime of $\cO_M$ of residue characteristic $l$, then a 
\emph{divisible $\cO_{M,\frakl}$-module} over a scheme $S$ will 
mean a $l$-divisible group $\cG/S$ equipped with a ring homomorphism 
$\cO_{M,\frakl} \to \End_S \cG$.

\begin{prop}\label{findAV}
	Let $k$ be algebraically closed of characteristic $l$, let $\frakl$ be a 
	prime of $M$ lying over $l$. Let $(\cG/k, \lambda)$ be a divisible 
	$\cO_{M,\frakl}$-module of height 
	$2[M_{\frakl}:\Ql]$ equipped with an $\cO_{M,\frakl}$-linear symmetric 
	isomorphism (i.e.~a principal quasi-polarization)
	$\lambda: \cG \cong \cG^{\vee}$.  
	
	Let $\mathfrak{c} \supseteq \cO $ be a fractional ideal of $M$ such that $l$ is 
	coprime to $\mathfrak{c}^{-1}$. Then there exists an 
	$M$-HBAV over $k$ equipped with $\mathfrak{c}$-polarization $(A,\iota,j)$ 
	and 
	an 
	isomorphism $i: A[\frakl^{\infty}] \cong \cG$ compatible with the 
	$\cO_{M,\frakl}$ actions on both sides such that $i^\vee\circ\lambda\circ 
	i = j(1)$.
\end{prop}
\begin{proof}The statement is essentially \cite[Thm.~7.4]{yuhilbert}, we just need to 
    take a little care to ensure that our $M$-HBAV has a 
    $\mathfrak{c}$-polarization. By 
    \cite[Thm.~7.3(1)]{yuhilbert} (see also \cite[Cor.~5.4.12]{gorenoort} 
    and
\cite{goren} for the case with $l$ unramified in $M$, which will suffice 
    for our applications) there is an $M$-HBAV with 
    $\mathfrak{c}'$-polarization (for some fractional ideal $\mathfrak{c}'$ of 
    $M$)
    $(A_0,\iota_0,j_0)$ such that the isocrystals of $A_0[\frakl^\infty]$ and 
    $\cG$ (ignoring the polarization and the $\cO_{M,\frakl}$-action) are 
    isomorphic.  We let $\mathfrak{b}$ be an (integral) ideal of $\cO_M$ with 
    $[\mathfrak{b}\mathfrak{c}'] = [\mathfrak{c}]$ in the narrow 
    class group of $M$. Replacing $A_0$ with $A_0/A_0[\mathfrak{b}]$ we obtain 
    a $\mathfrak{c}$-polarized $M$-HBAV $A_0$ such that the isocrystals of 
    $A_0[\frakl^\infty]$ and $\cG$ are isomorphic. 
    
    It follows from 
    \cite[Cor.~3.7]{yuhilbert} that the 
    quasi-polarized isocrystals with $\cO_{M,\frakl}$-action arising from the 
    two quasi-polarised divisible $\cO_{M,\frakl}$-modules 
    $(A_0[\frakl^\infty],\iota_0,j_0(1))$, $(\cG,\lambda)$ are isomorphic. We 
    can also fix choices of principally quasi-polarized 
    divisible $\cO_{M,\frakl'}$-modules for places $\frakl \ne \frakl'|l$ and 
    demand that $A_0[\frakl'^\infty]$ has quasi-polarized isocrystal isomorphic 
    to these. 
    
    By 
    Dieudonn\'{e} theory, we have an $\cO_M$-linear isogeny 
    $A_0 \overset{\pi}{\to} A$ with kernel contained in $A_0[l^n]$ for some 
    $n$ and a symmetric $\cO_M$-linear isogeny $\lambda_{A}: A \to 
    A^\vee$ with degree prime to $l$ (since we ensure it induces our principal 
    quasi-polarizations on our fixed divisible $\cO_{M,\frakl'}$-modules for 
    all places $\frakl'|l$) such that $\pi^\vee\circ\lambda_{A}\circ\pi = 
    l^{2n}j_0(1)$, together with an isomorphism $i: A[\frakl^{\infty}] 
    \to \cG$ such that $i^\vee\circ\lambda\circ i = \lambda_A$. Since 
    $\lambda_A$ has degree prime to $l$ and $\pi$ has $l$-power degree, it 
    follows from the equation $\pi^\vee\circ\lambda_{A}\circ\pi = 
    l^{2n}j_0(1)$ that $\lambda_A$ is a polarization with kernel 
    $A[\mathfrak{c}^{-1}]$. Applying Remark \ref{polnremark}, we obtain the 
    desired $M$-HBAV $A$ equipped with a $\mathfrak{c}$-polarization.    
\end{proof}
\begin{rem}\label{trivialpolarizationremark}
	If we have $\mathfrak{c}$ and $\cG$ as in the above Proposition, the map $x 
	\mapsto x\otimes 1$ induces an isomorphism $\cG \cong 
	\cG\otimes_{\cO_{M,\frakl}} 
	\mathfrak{c}_{\frakl}$ (it induces an isomorphism from $\cG/\cG[\mathfrak{c}_{\frakl}^{-1}]$ 
	to $\cG\otimes_{\cO_{M,\frakl}}\mathfrak{c}_{\frakl}$, and  
	$\mathfrak{c}^{-1}$ 
is coprime 
	to $l$). The quasi-polarization 
	$\lambda: \cG\cong \cG^\vee$ 
	therefore corresponds to an isomorphism $j_{\cG}:\cG\otimes_{\cO_{M, \frakl}}\mathfrak{c}_{\frakl} 
	\cong \cG^\vee$. The condition that $j(1) = i^\vee\circ \lambda\circ i$ 
	implies 
	that, under the isomorphism $i: A[\frakl^\infty] \cong \cG$, $j_{\cG}$  is 
	induced by the $\mathfrak{c}$-polarization $j$ on $A$.
\end{rem}

\begin{lemma}\label{FLlemma}
Let $l$ be an odd prime and let $v$ and $\frakl$ be primes of $F$ and $M$, respectively, 
unramified over $l$. 
Let $\rbar : G_{F_v} \to \GL_2(k_\frakl)$ be a continuous representation such that:
\begin{itemize}
 \item $\det\rbar = \overline{\epsilon}_l$,
 \item there is a crystalline lift $r : G_{F_v} \to \GL_2(\cO)$ (for a finite 
 extension $\cO/\cO_{M,\frakl}$) with labelled Hodge--Tate weights all equal to $\{-1,0\}$.
\end{itemize}
Let $V_{\rbar}$ be the $k_\frakl$-vector space scheme over $F_v$ underlying $\rbar$.
Then we can find a divisible $\cO_{M,\frakl}$-module $\cG$ defined over $\cO_{F_v}$ of height $2[\cO_{M,\frakl}:\Z_l]$ 
equipped with a $\cO_{M,\frakl}$-linear symmetric isomorphism $\lambda : \cG \cong \cG^\vee$, 
and an isomorphism $i : V_{\rbar} \cong \cG[\frakl]_{F_v}$ such that 
$i^\vee \circ \lambda[\frakl]_{F_v}\circ i$ is the isomorphism $V_{\rbar} \cong V_{\rbar}^\vee$ 
induced by the standard symplectic pairing on $V_{\rbar}$.
\end{lemma}

\begin{proof}
This follows from Fontaine--Laffaille theory \cite{fl}. 
First, since $l$ is unramified in $F_v$, the crystalline lift assumption 
implies that $\rbar$ is in the image of the Fontaine--Laffaille functor: 
using the notation of \textit{loc. cit.}, 
there is a $k_\frakl$-object $\overline{M}$ of $\mathrm{MF}_{tor}^{f,2}$ 
such that the action of $G_{F_v}$ on $U_S(\overline{M})$ 
is isomorphic to $\rbar$. 
By \cite[Lemma~2.4.1]{cht}, we can find a lift $r' : G_{F_v} \to \GL_2(\cO_{M,\frakl})$ 
of $\rbar$ such that for each $n\ge 1$, 
there is an $\cO_{M,\frakl}$-object $M_n\in \mathrm{MF}_{tor}^{f,2}$ such that the action 
of $G_{F_v}$ on $U_S(M_n)$ is isomorphic to $r' \bmod l^n$. 
(\textit{Loc. cit.} uses a covariant version of the functor $U_S$, 
but the proof shows that the Fontaine--Laffaille modules can be 
deformed through Artinian thickenings, so carries over unchanged.) 
Then $r$ is crystalline with all labelled Hodge--Tate weights 
equal to $\{-1,0\}$, so $\det r'|_{I_K} = \epsilon_l$. 
Since $\det\rbar = \epsilon_l$ and $l>2$, we can find an unramified 
character $\psi : G_{F_v} \to 1+l\cO_{M,\frakl}$ such that 
$\psi^2 = (\det r')\epsilon_l^{-1}$ and $r'':= r'\otimes\psi$ 
is a lift of $\rbar$ with determinant $\epsilon_l$. 
Moreover, there are $\cO_{M,\frakl}$-objects $N_n$ of $\mathrm{MF}_{tor}^{f,2}$ 
corresponding to $\psi \bmod l^n$ for each $n\ge 1$, 
and the action of $G_{F_v}$ on $U_S(M_n \otimes N_n)$ is given by $r'\otimes \psi \bmod l^n$. 
Applying \cite[\S9.11 and Proposition~9.12]{fl} to the collection $\{M_n \otimes N_n\}_{n \ge 1}$, 
we obtain a divisible $\cO_{M,\frakl}$-module $\cG$ defined over $\cO_{F_v}$ 
such that the $G_{F_v}$-action on the Tate module 
$T_l(\cG)$ is isomorphic to $r''$. 
In particular, we have an isomorphism $i : V_{\rbar} \cong \cG[\frakl]_{F_v} = \cG[l]_{F_v}$ 
of $k_\frakl$-vector space schemes over $F_v$.

It remains to produce $\lambda$.
Since $\det r'' = \epsilon_l$, letting $T = \cO_{M,\frakl}^2$ with $G_{F_v}$-action by $r''$, 
the standard symplectic pairing on $T$ composed with the 
trace pairing $\cO_{M,\frakl} \otimes \cO_{M,\frakl} \to \Z_l$ gives an isomorphism 
$T \cong \Hom_{\Z_l}(T,\Z_l(1))$. 
This implies $T_l(\cG) \cong T_l(\cG^\vee)$ compatibly with the 
$\cO_{M,\frakl}$-module structure.
By a theorem of Tate \cite[Theorem~4]{Tatepdiv}, 
we obtain a $\cO_{M,\frakl}$-linear symmetric isomorphism $\lambda : \cG \cong \cG^\vee$ 
such that $i^\vee \circ \lambda[\frakl]_{F_v}\circ i$ is the isomorphism $V_{\rbar} \cong V_{\rbar}^\vee$ 
induced by the standard symplectic pairing on $V_{\rbar}$.
\end{proof}

\begin{thm}\label{potaut} 
	Suppose $F$ is a CM field, $l$ is an odd prime which is unramified in $F$ 
	and 
	we 
	have a continuous absolutely
	irreducible representation 
	\[\rhobar: G_F \to \GL_2(k) \] with $k/\Fl$ finite such that:
	\begin{itemize}
	\item $\det \rhobar = \barepsilon_l^{-1}$
	\item For all $v|l$, $\rhobar|_{G_{F_v}}$ has a crystalline lift $\rho_v: 
	G_{F_v} \to 
	\GL_2(\cO)$ (for a finite extension $\cO/W(k)$) 
	with labelled Hodge--Tate weights all equal to 
	$\{0,1\}$
\end{itemize}
	
	Suppose moreover that $F^\avoid/F$ is a finite extension. Then we 
	can find a finite CM extension $F_1/F$, linearly disjoint from 
	$F^\avoid$ over $F$ and with $l$ unramified in $F_1$, a regular algebraic 
	cuspidal automorphic 
	representation $\pi$ for $\GL_2(\A_{F_1})$ unramified at places above $l$ 
	and of weight $0$, together with an 
	isomorphism $\iota: \Qlbar \iso \C$ such that (composing $\rhobar$ with 
	some embedding $k \hookrightarrow \Flbar$) \[\rbar_{\iota}(\pi) \cong 
	\rhobar|_{G_{F_1}}.\]
	
	If $\overline{v}_0\nmid l$ is a finite place of $F^+$,
then we can moreover find $F_1$ and $\pi$ as above with $\pi$ 
	unramified above $\overline{v}_0$.	
\end{thm}
 \begin{proof}

We begin by choosing a totally real number field $M$ together with a prime 
$\frakl | l$ of $M$ such that $l$ is unramified in $M$ and $k_\frakl$ is 
isomorphic to $k$. We fix an isomorphism $k \cong k_\frakl$ and regard 
$\rhobar$ as a representation with coefficients in $k_\frakl$.

Choose a non-CM elliptic curve $E/\Q$ with good reduction at $l$ and the 
rational prime $q$ under $\overline{v}_0$. Choose a rational prime $p \ne l$ 
such that
\begin{itemize}
\item $p > 5$ splits completely in $FM$,
\item $\SL_2(\Fp) \subset \rbar_{E,p}(G_F)$, $E$ 
has good reduction at $p$ and $\rhobar$ is unramified at places dividing $p$.
\end{itemize}
The second condition is satisfied by all but finitely many primes and the first 
condition is satisfied by a positive density set of primes, so we can find such 
a $p$. We fix a prime $\gp | p$ of $M$.

Let $V_{\rhobar}^\vee$ denote the $k_{\frakl}$-vector space scheme over $F$ 
underlying the dual representation $\rhobar^\vee$ and 
fix the standard symplectic pairing on it, which $\rhobar^\vee$ will preserve 
up to 
multiplier $\barepsilon_l$. We also have the $k_{\gp} \cong \F_p$-vector 
space scheme $(E\otimes_{\Z} \cO_M)[\gp] \cong E[p]$
over $F$, which comes equipped with the Weil pairing.

Denoting the inverse different of $M$ by $\cD^{-1}$, we let $Y$ be the scheme 
over $F$ classifying tuples $(A,j,\alpha_{\rhobar}, 
\alpha_E)$ where:

\begin{itemize}
	\item $A$ is an $M$-HBAV with $\cD^{-1}$-polarization $j$
	\item $\alpha_{\rhobar} : A[\frakl] \to V_{\rhobar}^\vee$ and $\alpha_E: 
	A[\gp] 
	\to E[p]$ are isomorphisms of vector space schemes compatible with our 
	fixed symplectic pairings on the right hand sides and with the pairings 
	(see (\ref{pairing}))
	$A[\frakl] \times A[\frakl] \to (\cO_M/\frakl)(1)$ and $A[\gp] \times 
	A[\gp] 
	\to (\cO_M/\gp)(1)$ on 
	the left hand sides.
\end{itemize}

As in \cite{tay-fm2}, $Y/F$ is a smooth, geometrically connected variety. 
We let $X$ be the restriction of scalars $X = \Res_{F/F^+}Y$, which is also 
smooth and geometrically connected.

Now we apply Proposition \ref{ecpotaut} with $\cL = \{l,p\}$ and $L_1^{\avoid}$ 
the normal closure of $F^{\avoid}\overline{F}^{\ker(\rhobar \times 
\rbar_{E,p})}$ over $\Q$. We obtain a finite Galois extension $L_2^{\avoid}/\Q$ 
linearly disjoint from $L_1^{\avoid}$ over $\Q$ and a finite totally real Galois 
extension $L^{\suffi}/\Q$ which is unramified above $p$ and $l$ and linearly 
disjoint from $L_1^{\avoid}L_2^{\avoid}$ over $\Q$. 

We are going to apply Proposition \ref{moretbailly} to 
$X$ with the following input data:

\begin{itemize}
	\item $L = F^+$, $S_1 = \{\bar{v}|lp\}$, $S_2 = \{\bar{v}_0\}$, $L^{\avoid} 
	= L_1^{\avoid}L_2^{\avoid}L^{\suffi}$
	\item for $\bar{v}|lp$, $\Omega_{\bar{v}} \subset X((F^+_{\bar{v}})^{\nr}) 
	= 
	\prod_{v|\bar{v}}Y(F_v^{\nr})$ is the subset given by Abelian varieties $A$ 
	with good reduction at $v$
	\item $\Omega_{\bar{v}_0} \subset  X(\overline{F^+_{\bar{v}_0}}) = 
	\prod_{v_0|\bar{v}_0}Y(\overline{F}_{v_0})$ is the subset given by Abelian 
	varieties $A$ 
	with good reduction at $v_0$.
\end{itemize}

We need to check that the various hypotheses of Proposition \ref{moretbailly} 
are satisfied. It is clear that $X(F^+_{\bar{v}}) = Y(F_v)$
is non-empty for the real places $\bar{v}$ of $F^+$ ($v$ denotes the unique 
complex place of $F$ extending $\bar{v}$).

For $v$ a place of $F$ 
dividing $p$, we can find a positive integer $f$ such that 
$\rhobar(\Frob_v)^{-f}$ and $\rbar_{E,l}(\Frob_v)^f$ are trivial. We can then 
take $A$ to be the base change of $E\otimes_{\Z}\cO_M$ to the unramified degree 
$f$ extension of $F_v$, $j$ to be induced by the Weil pairing on $E$, 
$\alpha_E$ to be the canonical identification (recall that $p$ splits 
completely in $M$) and $\alpha_{\rhobar}$ to be an isomorphism compatible with 
the Weil pairing on $A[\frakl]$ and our fixed pairing on $V_{\rhobar}^\vee$. 
This shows that for $\bar{v}|p$, $\Omega_{\bar{v}}$ is non-empty. A similar 
argument applies to $\Omega_{\bar{v}_0}$; we can work over an extension which 
trivialises $\rhobar|_{G_{F_{v_0}}}$ for $v_0|\bar{v}_0$.

It remains to handle the case of $v|l$; we set $K = F_v$. By Lemma 
\ref{FLlemma}, we have a divisible $\cO_{M,\frakl}$-module $\cG$ over 
$\cO_K$ equipped with a principal quasi-polarization $\lambda: \cG\cong 
\cG^\vee$ such that the $G_K$ action on $\cG[\frakl]_K$ is isomorphic to 
$\rhobar^\vee$ and $\lambda$ induces our fixed 
pairing on $V_{\rhobar}^\vee$.
We can work with an 
integral model $\cY/\cO_K$ for $Y_K$, classifying tuples 
$(A,j,\alpha_{\rhobar},\alpha_E)$, where now $A/S$ ($S$ an $\cO_K$-scheme) is 
an $M$-HBAV with $\cD^{-1}$-polarization $j$ and $\alpha_{\rhobar}: A[\frakl] 
\to \cG[\frakl]$ is an 
isomorphism of vector space 
schemes, compatible with the isomorphisms $A[\frakl] \cong A[\frakl]^\vee$ 
induced by $j$ (see Remark \ref{trivialpolarizationremark}) and $\lambda: 
\cG[\frakl]\cong \cG[\frakl]^\vee$ and similarly 
for $\alpha_E$ ($E$ 
has good reduction at $l$, so $E[p]$ extends to a vector space scheme 
over $\cO_K$ equipped with a canonical isomorphism $E[p] \cong E[p]^\vee$). Now 
it suffices to show that 
$\cY(\cO_K^{\nr})$ is non-empty. In 
fact, by Greenberg's approximation theorem \cite[Corollary 
2]{greenbergapprox}, it suffices to show that $\cY(\breve{\cO})$ is non-empty, 
where 
$\breve{\cO}$ is the $l$-adic completion of $\cO_K^{\nr}$. It follows from 
  Proposition \ref{findAV} that we have a $\cD^{-1}$-polarized $M$-HBAV 
  $(A_1,j)$
  over $k = \cO_K^{\nr}/\frakl$ with $\frakl$-divisible module isomorphic to 
  $\cG_k$ and $j(1)$ inducing our fixed 
  quasi-polarization on $\cG_k$. By  
Serre--Tate deformation theory, we can lift $A_1$ 
to an Abelian scheme $\widetilde{A}_1$ over $\breve{\cO}$ equipped with a 
$\cD^{-1}$-polarization $\widetilde{j}$ and an isomorphism  
$\widetilde{A}_1[\frakl^{\infty}] \iso \cG_{\breve{\cO}}$ under which 
$\widetilde{j}$ corresponds to $\lambda$. In particular, the induced
isomorphism $\alpha_{\rhobar}:\widetilde{A}_1[\frakl] \iso 
\cG_{\breve{\cO}}[\frakl]$ is 
compatible with the (quasi-)polarizations on both sides. This gives us the
$\widetilde{A}_1, j$ and $\alpha_{\rhobar}$ we need. 
We let $\alpha_E$ be an 
isomorphism (between two trivial vector space schemes) compatible with the 
polarizations on each side. Now we have described a point of $\cY(\breve{\cO})$ 
as 
desired.

We have checked the hypotheses of Proposition \ref{moretbailly}. So we 
obtain a finite Galois totally real extension $F^+_0/F^+$, linearly disjoint 
from $L_1^{\avoid} L_2^{\avoid} L^{\suffi}$ over $F^+$ (and in particular from 
$F$, so $F_0:=F^+_0F$ is a totally imaginary quadratic extension of $F^+_0$) 
and a point 
$(A,j,\alpha_{\rhobar},\alpha_E)$ of $X(F_0^+)$ such that $A$ has good reduction 
above $\bar{v}_0lp$. Moreover, $l$ and $p$ are unramified in $F_0$. 

Finally, we set $F_1 = F^+_0L^{\suffi}F$, a CM extension of $F$ which is 
unramified above $p$ and $l$.  Since $F_1^+$ is linearly disjoint from 
$L_2^{\avoid}$ over $\Q$ and contains $L^{\suffi}$, Proposition \ref{ecpotaut} 
tells us that there is a regular algebraic conjugate self-dual cuspidal 
automorphic representation $\sigma$ of $\GL_2(\A_{F_1})$ of weight $0$ such 
that $\rbar_{\iota}(\sigma) \cong \rbar_{E,p}^\vee|_{G_{F_1}}$ for all $\iota: 
\Qpbar \iso \C$. Moreover we can assume that $\sigma$ is unramified above 
$\bar{v}_0lp$. Since $F_1$ is linearly disjoint from $L_1^{\avoid}$ over $F$, we 
have $\SL_2(\Fp) \subset \rbar_{E,p}(G_{F_1})$.  

Fixing a choice of $\iota$ and applying Theorem 
\ref{thm:main_automorphy_lifting_theorem}, we deduce 
that we have a regular algebraic cuspidal automorphic representation $\pi$ of 
$\GL_2(\A_{F_1})$, unramified at places above $pl\bar{v}_0$ and of weight $0$ 
such that $r_{A,\frakp}^\vee \cong r_\iota(\pi)$. Our choice of $\iota$ 
determines an embedding $\tau: M \hookrightarrow \C$ by composing $\iota$ with 
$M 
\hookrightarrow M_\frakp = \Qp$. We choose $\iota_l: \Qlbar \iso \C$ so that 
the embedding $\iota_l^{-1}\circ\tau$ induces the place $\frakl$, and denote 
the induced embedding $M_\frakl \hookrightarrow \Qlbar$ by 
$\iota_{M_\frakl}$.
It follows that we have $r_{\iota_l}(\pi) \cong 
\iota_{M_\frakl} \circ r_{A,\frakl}^\vee$, and we deduce the statement of the 
theorem since we have an 
isomorphism $\alpha_{\rhobar}: A[\frakl] \cong V_{\rhobar}^\vee$.

\end{proof}

\section{Local--global compatibility}
\begin{theorem}\label{thm:monodromynew}
Suppose that $F$ is a CM field and that $\pi$ is a regular
algebraic cuspidal automorphic representation of $\GL_2(\A_F)$ of
weight $0$, and let $M_\pi \subset \C$ be its coefficient field.
Let $\lambda|l$ be a prime of $M_\pi$ such that:
\begin{enumerate}
    \item $l \ge 5$, $l$ is unramified in $F$, and lies under no prime at which 
    $\pi$ is ramified.
    \item $\overline{r}_{\pi,\lambda}$ is decomposed generic, 
    $\overline{r}_{\pi,\lambda}(G_{F(\zeta_l)})$ is enormous, 
    and there is a $\sigma \in G_F - G_{F(\zeta_l)}$ such that 
    $\overline{r}_{\pi,\lambda}(\sigma)$ is scalar.
\end{enumerate}

Then, for any $\iota : \overline{M}_{\pi,\lambda} \iso \C$ and 
any finite $v\nmid l$ in $F$, we have
\[ \iota \mathrm{WD}(r_{\pi,\lambda}|_{G_{F_v}})^{\Fss} \cong 
\rec_{F_v}(\pi_v \lvert\det\rvert^{-1/2}). \] 
\end{theorem}
\begin{proof}
Fix a prime $p \ne l$ for which $\rbar_{\pi,\lambda}$ is decomposed generic. 
By the main result of \cite{ilavarma}, to prove the theorem it suffices to 
show that if $v\nmid l$ is a finite place at which $\pi$ is special, then $r_{\pi,\lambda}$ has nontrivial monodromy at $v$.
Fix $\iota : \overline{M}_{\pi,\lambda} \iso \C$ and 
let $N$ be the monodromy operator 
for $\mathrm{WD}(r_{\pi,\lambda}|_{G_{F_v}})^{\Fss}$. 
To show $N \ne 0$, it suffices to do so after restriction to any finite 
extension. 
In particular, making a solvable base change that is disjoint from 
$\overline{F}^{\ker(\overline{r}_{\pi,\lambda})}$ in which $l$ is unramified and 
$p$ is totally split, we may assume that 
\begin{itemize}
\item $\pi_v$ is an unramified twist of the Steinberg representation, 
\item $\overline{r}_{\pi,\lambda}$ is unramified at $v$ and $v^c$.
\end{itemize}

Now assume for a contradiction that $N = 0$. 
Then the main result of \cite{ilavarma} implies that 
$r_{\pi,\lambda}|_{G_{F_v}} \cong \chi \oplus \chi\epsilon_l$ 
for an unramified character $\chi : G_{F_v} \to \overline{M}_{\pi,\lambda}^\times$. 
Now we apply Theorem \ref{potaut} with $F^{\mathrm{avoid}}$ equal to the Galois 
closure of $\overline{F}^{\ker(\overline{r}_{\pi,\lambda})}(\zeta_l)/\Q$, to 
obtain a CM Galois extension $F_1/F$, 
linearly disjoint from $F^{\mathrm{avoid}}$ over $F$ and with $l$ 
unramified in $F_1$, such that $\overline{r}_{\pi,\lambda}|_{G_{F_1}}$ is 
automorphic (coming from a weight 0, unramified above $v$ and $l$, 
automorphic representation). 
We now wish to apply Theorem~\ref{thm:main_automorphy_lifting_theorem}. 
By our choice of $F^{\mathrm{avoid}}$, it is easy to see that 
$\overline{r}_{\pi,\lambda}(G_{F_1(\zeta_l)})$ is enormous and 
that there is $\sigma \in G_{F_1} - G_{F_1(\zeta_l)}$ such that 
$\overline{r}_{\pi,\lambda}(\sigma)$ is scalar. 
We claim that $\overline{r}_{\pi,\lambda}|_{G_{F_1}}$ is also decomposed 
generic. 

Let $\widetilde{F}$ and $\widetilde{F}_1$ be the Galois closures of 
$F/\Q$ and $F_1/\Q$, respectively.  
Since $F^\avoid/\Q$ is Galois and $F^\avoid \cap F_1 = F$, 
we have $F^\avoid \cap \widetilde{F}_1 = \widetilde{F}$.
Since $p$ is totally split in $F$, it is totally split in $\widetilde{F}$ and 
the conjugacy class of $\Frob_p$ in $\Gal(F^\avoid/\Q)$ lies in $\Gal(F^\avoid/\widetilde{F})$. 
Using Chebotarev density, we choose a prime $q$ unramified in $F^\avoid\widetilde{F}_1$ 
such that $\Frob_q \in \Gal(F^\avoid\widetilde{F}_1/\Q)$ lies in 
$\Gal(F^\avoid\widetilde{F}_1/\widetilde{F})$ and corresponds to $\Frob_p \times 1$ under the isomorphism 
\[ \Gal(F^\avoid\widetilde{F}_1/\widetilde{F}) \cong 
\Gal(F^\avoid/\widetilde{F}) \times \Gal(\widetilde{F}_1/\widetilde{F}). \]
This $q$ is decomposed generic for 
$\overline{r}_{\pi,\lambda}|_{G_{F_1}}$.

By Theorem~\ref{thm:lg-at-p}, $r_{\pi,\lambda}|_{G_{F_1}}$ is crystalline 
with all labelled Hodge--Tate weights equal to $\{0,1\}$ 
at all places above $l$ in $F_1$.  
The representation $r_{\pi,\lambda}|_{G_{F_1}}$ thus satisfies the assumptions of 
Theorem~\ref{thm:main_automorphy_lifting_theorem}, 
and we obtain a regular algebraic cuspidal automorphic representation $\Pi$ of 
$\GL_2(\A_{F_1})$ such that 
$r_{\pi,\lambda}|_{G_{F_1}} \cong r_{\iota}(\Pi)$ 
and with $\Pi_w$ unramified at all $w|v$ in $F_1$.
Then for any $w|v$ in $F_1$, $ r_{\iota}(\Pi)|_{G_{F_{1,w}}} \cong \chi|_{G_{F_{1,w}}} \oplus 
\chi|_{G_{F_{1,w}}}\epsilon_l$, and $\Pi_w$ is an unramified principal series. 
By local-global compatibility at unramified places \cite{hltt,ilavarma}, 
this contradicts the genericity of $\Pi$.
\end{proof}

\begin{proof}[Proof of Theorem~\ref{theorem_main_from_intro}]
If $\pi$ is everywhere potentially unramified, then this follows from the main result of \cite{ilavarma}, 
so we can assume that $\pi$ is special at some finite place of $F$. 
Theorem \ref{theorem_main_from_intro} then follows at once from Theorem~\ref{thm:monodromynew} and Lemma~\ref{lemma:bad_set}.
\end{proof}

\emergencystretch=3em

\bibliographystyle{amsalpha}
\bibliography{CMpatching}

\end{document}